\documentclass[11pt]{amsart}
\usepackage{enumerate,amscd,amsfonts,amssymb}

\newcommand{\e}{\varepsilon}
\newcommand{\IR}{\mathbb R}
\newcommand{\dens}{\mathrm{dens}}
\newcommand{\w}{\omega}
\newcommand{\PMetr}{\mathbf{PMetr}}
\newcommand{\Dist}{\mathbf{PMetr}}
\newcommand{\Metr}{\mathbf{Metr}}
\newcommand{\Top}{\mathbf{Top}}
\newcommand{\Dec}{\mathsf{D}}
\newcommand{\seq}{\mathsf{s}}
\newcommand{\Set}{\mathbf{Set}}
\newcommand{\cl}{\mathrm{cl}}
\newcommand{\Eco}{\mathsf{E}}

\newcommand{\Cob}{\circledast}
\newcommand{\IN}{\mathbb N}
\newcommand{\IZ}{\mathbb Z}
\newcommand{\IA}{\mathrm{II}}

\newcommand{\pr}{\mathrm{pr}}
\newcommand{\N}{\ensuremath{\mathbb{N}}}

\newcommand{\Ra}{\Rightarrow}
\newcommand{\la}{\langle}
\newcommand{\ra}{\rangle}
\newcommand{\id}{\mathrm{id}}

\newtheorem{theorem}{Theorem}[section]
\newtheorem{lemma}[theorem]{Lemma}
\newtheorem{corollary}[theorem]{Corollary}

\newtheorem{proposition}[theorem]{Proposition}
\newtheorem{problem}[theorem]{Problem}
\theoremstyle{definition}
\newtheorem{remark}[theorem]{Remark}
\newtheorem{example}[theorem]{Example}
\newtheorem{definition}[theorem]{Definition}

\title{Connected economically metrizable spaces}
\author{T.~Banakh, M.~Vovk, M.~R.~W\'ojcik}
\address{Taras Banakh: Department of Mathematics, Ivan Franko National University of Lviv, Ukraine, and\newline
Instytut Matematyki, Unwersytet Humanistyczno-Przyrodniczy Jana Kochanowskiego, Kielce, Poland}
\email{tbanakh@yahoo.com}
\address{M.~Vovk: National University ``Lvivska Politechnika", Lviv, Ukraine}
\address{Micha{\l} Ryszard W\'ojcik:\newline Department of Mathematics, University of Louisville, Louisville, USA}
\email{michal.ryszard.wojcik@gmail.com}
\subjclass[2000]{54B30;  54C30; 54D05; 54E35; 54E50; 54F15; 54G15; 54G20}
\keywords{Nonseparably connected complete metric space, economical metric space, quotient map, monotone map}

\begin{document}

\maketitle

\begin{abstract}
A topological space is {\em nonseparably connected} if it is connected
but all of its connected separable subspaces are singletons. 
We show that each connected sequential topological space $X$ is the image of a nonseparably connected complete metric space
$\Eco X$ under a monotone quotient map. The metric $d_{\Eco X}$ of the space $\Eco X$ is 
{\em economical} in the sense that
for each infinite subspace $A\subset X$
the cardinality of the set $\{d_{\Eco X}(a,b):a,b\in A\}$
does not exceed the density of $A$,
$|d_{\Eco X}(A\times A)|\le\dens(A)$. 

The construction of the space $\Eco X$ determines a functor $\Eco:\Top\to\Metr$ from the category $\Top$ of topological spaces and their continuous maps into the category $\Metr$ of metric spaces and their non-expanding maps.
\end{abstract}



\section{Introduction}

This paper was motivated by the problem of constructing a nonseparably connected complete metric space, posed in \cite{MWp} and \cite{MW}. A topological space $X$ is called {\em separably connected} if any two points of $X$ lie in a connected separable subspace.
On the other hand, a topological space $X$ is {\em nonseparably connected} if it is connected but all connected separable subspaces of $X$ are singletons. 

The first example of a nonseparably connected metric space was constructed by R.Pol
in 1975, \cite{Pol}. Another example was given by P.Simon \cite{Sim} in 2001.  
In 2008, Morayne and W\'ojcik obtained a nonseparably
connected metric group as a graph of an additive function from the real line to a non-separable Banach space, see \cite{WPhD} or \cite{MW}. None of these nonseparably connected spaces is completely metrizable.

In this paper we shall suggest a general (functorial) construction of nonseparably connected complete metric spaces. 

Our approach is based on the notion of an economical metric space. This is a metric space $(X,d)$ such that for each infinite subspace $A\subset X$ the set $d(A\times A)=\{d(a,b):a,b\in A\}$ has cardinality $|d(A\times A)|\le\dens(A)$ where $\dens(A)=\min\{|D|:D\subset A\subset\overline{D}\}$ stands for the density of $A$.

It is easy to see that each separable subspace of an economic metric space is zero-dimensional and hence each connected economically metrizable space is nonseparably connected. The following theorem, which is the main result of the paper, yields us many examples of connected economical complete metric spaces, thus resolving Problem 2 of \cite{MW}. This theorem is proved in Section~\ref{ecorez}, see Theorem~\ref{t9.1}.

\begin{theorem}\label{t1.1} Each (connected) sequential topological space $X$ is the image of a (connected) economical complete metric space $\Eco X$ under a monotone quotient map $\xi_X:\Eco X\to X$. 
\end{theorem}

As we shall see, the construction of the space $\Eco X$ determines a functor $$\Eco:\Top\to\Metr$$ from the category $\Top$ of topological spaces and their continuous maps to the category $\Metr$ of metric spaces and their non-expanding maps.
The functor $\Eco$ will be defined as the composition $\Eco=\Cob^\w\circ \Dec$ of the functor of sequence decomposition $\Dec:\Top\to\PMetr$ and the functor of infinite iterated cobweb $\Cob^\w:\PMetr\to\Metr$. Here $\PMetr$ is the category of premetric spaces and their non-expanding maps. A {\em premetric space} is a pair $(X,d)$ consisting of a set $X$ and a function $d:X\times X\to[0,\infty)$ such that $d(x,x)=0$ for all $x\in X$. Premetric spaces will be considered in Section~\ref{premetric}, the functors $\Dec$ and $\Cob^\w$ are defined in Sections~\ref{dec} and \ref{cobweb}, respectively.

\section{Economical metric spaces}\label{economic}

We recall that the metric $d$ of a metric space $(X,d)$ is {\em economical} if for each infinite subset $A\subset X$ the set $d(A\times A)=\{d(a,b):a,b\in A\}$ has cardinality $|d(A\times A)|\le\dens(A)$.

Typical examples of economical metric spaces are ultrametric spaces. We recall that a metric $d$ on a set $X$ is called an {\em ultrametric} if it satisfies the strong triangle inequality
$$d(x,z)\le\max\{d(x,y),d(y,z)\}\mbox{ for all $x,y,z\in X$}.
$$

\begin{proposition}\label{ultra} Each ultrametric space is economical.
\end{proposition}

\begin{proof} We should check that $|d(A\times A)|\le\dens(A)$ for any infinite subset $A$ of an ultrametric space $(X,d)$. Assuming that $|d(A\times A)|>\dens(A)$, we conclude that the set $D=d(A\times A)\setminus\{0\}$ has cardinality $|D|>\dens(A)$. For every $t\in D$ select a pair of points $x_t,y_t\in A$ with $d(x_t,y_t)=t$.
Since the subspace $B=\{(x_y,y_t)\}_{t\in D}\subset A\times A$  has cardinality $|B|=|D|>\dens(A)=\dens(A\times A)\ge\dens(B)$, it is not discrete and hence has a non-isolated point $(x_t,y_t)\in B$. Then we can find $s\in D\setminus\{t\}$ such that the point $(x_s,y_s)$ is so near to $(x_t,y_t)$ that 
$$\max\{d(x_s,x_t),d(y_s,y_t)\}<\tfrac13d(x_t,t_t)=\tfrac t3.$$
The triangle inequality for $d$ guarantees that
$$s=d(x_s,y_s)\ge d(x_t,y_t)-d(x_t,x_s)-d(y_t,y_s)>\tfrac13d(x_t,y_t)=\tfrac t3.$$
On the other hand, the strong triangle inequality implies that
$$s=d(x_s,y_s)\le\max\{d(x_s,x_t),d(x_t,y_t),d(y_y,y_s)\}=d(x_t,y_t)=t.$$
By the same reason,
$$t=d(x_t,y_t)\le\max\{d(x_t,x_s),d(x_s,y_s),d(y_s,y_t)\}=d(x_s,y_s)=s.$$
Unifying those inequalities, we get $t=s$, which contradicts the choice of $s$.
\end{proof}

Proposition~\ref{ultra} implies that the Cantor cube $2^\w=\{0,1\}^\w$ endowed with the ultrametric 
$$d((x_n),(y_n))=\max_{n\in\w}|x_n-y_n|/2^n$$is an economical metric space. Yet, the Cantor cube $2^\w$ is homeomorphic to the Cantor set $C\subset\IR$ which, being endowed with the Euclidean metric $d(x,y)=|x-y|$, is not economical. This justifies the following definition.

A topological space $X$ is defined to be {\em economically metrizable} if the topology of $X$ is generated by an economical metric. 

\begin{proposition}\label{eco-p1} If $X$ is an economically metrizable space, then each subspace $A\subset X$ of density $\dens(A)<\mathfrak c$ is zero-dimensional. Consequently, each connected economically metrizable space is nonseparably connected.
\end{proposition}

\begin{proof} Let $d$ be an economical metric generating the topology of $X$. Given any subspace $A\subset X$ of density $\dens(A)<\mathfrak c$, we conclude that  $|d(A\times A)|\le\dens(A)<\mathfrak c$ and hence the set $R=[0,\infty)\setminus d(A\times A)$ is dense in $(0,\infty)$. 
Since each ball $B_A(a,r)$, $a\in A$, $r\in R$, is open-and-closed in $A$, the space $A$ has a base of the topology consisting of closed-and-open set, which means that $A$ is zero-dimensional.
\end{proof} 

\begin{problem} Let $X$ be an economical metric space and $A\subset X$ be a subspace of density $\dens(A)<\mathfrak c$. Is $A$ strongly zero-dimensional?
\end{problem}

We recall that a metric space $X$ is {\em strongly zero-dimensional} if for any disjoint closed subsets $A,B\subset X$ there is an open-and-closed subset $U\subset X$ such that $A\subset U\subset X\setminus B$, see \cite[6.2.4]{En}.

We do not know if Proposition~\ref{eco-p1} can be reversed.

\begin{problem} Is a metrizable topological space $X$ economically metrizable if each subspace $A\subset X$ of density $\dens(A)<\mathfrak c$ is (strongly) zero-dimensional?
\end{problem} 

This problem can be posed more generally as:

\begin{problem} Characterize topological spaces whose topology is generated by an economical metric.
\end{problem}

We can also ask about the characterization of metrizable spaces $X$ such that any metric generating the topology of $X$ is economical. This question has the following answer:

\begin{theorem} For a metrizable topological space $X$ the following conditions are equivalent:
\begin{enumerate}
\item each metric generating the topology of $X$ is economical;
\item $|f(A)|\le\dens(A)$ for any subspace $A\subset X$ and a continuous map $f:X\to\IR$;
\item $\min\{|A|,\mathfrak c\}\le\dens(A)$ for any subspace $A\subset X$.
\end{enumerate}
\end{theorem}

\begin{proof} $(1)\Ra(2)$ Assume that each metric generating the topology of $X$ is economical and fix any metric $d$ generating the topology of $X$. Assume that $|f(A)|>\dens(A)$ for some continuous function $f:X\to \IR$ and some subspace $A\subset X$. In this case the set $A$ is infinite. Without loss of generality, $f(a)=0$ for some $a\in A$ and hence $$|f(A)|=\big|\{|f(x)|:x\in A\}\big|=\big|\{|f(x)-f(a)|:x\in A\}\big|.$$

It is easy to see that the metric $\rho$ on $X$ defined by
$$\rho(x,y)=d(x,y)+|f(x)-f(y)|$$
generates the topology of $X$. 

By our assumption, both metrics $d$ and $\rho$ are economical. Consequently the sets $\{d(x,a):x\in A\}$, $\{\rho(x,a):x\in A\}$ have cardinality $\le\dens(A)$. Then 
$$|f(A)|=|\{|f(x)-f(a)|:x\in A\}|=|\{\rho(x,a)-d(x,a):a\in A\}|\le\dens(A),$$which is a desired contradiction.
\smallskip

$(2)\Ra(3)$ 
Assume that $\min\{|A|,\mathfrak c\}>\dens(A)$ for some subset $A\subset X$ but $|f(A)|\le\dens(A)$ for any continuous function $f:X\to\IR$. Without loss of generality, $A$ is a closed subspace of $X$. We claim that the space $A$ is strongly zero-dimensional. Given two disjoint closed subsets $E,F\subset A$ we should find an 
open-and-closed subset $U\subset A$ such that $E\subset U\subset A\setminus F$.

By the normality of $X$ there is a continuous function $f:X\to [0,1]$ such that $f(E)\subset\{0\}$ and $f(F)\subset\{1\}$. By our hypothesis, the set $f(A)$ has cardinality $|f(A)|\le\dens(A)<\mathfrak c$. Consequently, we can find a number $t\in(0,1)\setminus f(A)$. Then $U=A\cap f^{-1}([0,t))$ is the required open-and-closed set in $A$ separating $E$ from $F$. 

By \cite[6.2.4, 7.3.15]{En}, the space $A$, being metrizable and strongly zero-dimensional, embeds into the countable power $D^\w$ of the discrete space $D$ of cardinality $|D|=\dens(A)<\mathfrak c$. Since the discrete space $D$ admits a continuous injective map into the Cantor cube $2^\w$, the countable power $D^\w$ also admits such a continuous injective map. Taking into account that the Cantor cube embeds into the real line, we conclude that the space $A$ admits a continuous injective map $g:A\to \IR$ to the real line. By Tietze-Urysohn Theorem~\cite[2.1.8]{En}, the map $g$ has a continuous extension $f:X\to\IR$. Then $|f(A)|=|g(A)|=|A|>\dens(A)$ and this is a contradiction.
\smallskip

The implication $(3)\Ra(1)$ is trivial.
\end{proof}

 \section{Premetric spaces}\label{premetric}

The definition of a metric on a set $X$ is well-known: this is a function $d:X\times X\to[0,\infty)$ satisfying four axioms:
\begin{enumerate}
\item $d(x,x)=0$,
\item $d(x,y)$ implies $x=y$,
\item $d(x,y)=d(y,x)$,
\item $d(x,z)\le d(x,y)+d(y,z)$,
\end{enumerate}
 for any points $x,y,z\in X$.

Deleting some of these axioms we obtain various generalizations of metrics: pseudometrics (they obey the axioms 1,3,4), quasimetrics (1,2,4), symmetrics (1,2,3). The most radical generalization of a metric is that of a premetric, see \cite[\S2.4]{AF}.

\begin{definition} A {\em premetric} on a set $X$ is any function $d:X\times 
 X\to[0,\infty)$ such that $d(x,x)=0$ for all $x\in X$. A {\em premetric space} is a pair $(X,d)$ consisting of a set $X$ and a premetric $d$ on $X$. In the sequel the premetric of a premetric space $X$ will be denoted by $d_X$ or just $d$ if the set $X$ is clear from the context. 
\end{definition}

A map $f:X\to Y$ between two premetric spaces is called {\em non-expanding} if $d_Y(f(x),f(y))\le d_X(x,y)$ for all $x,y\in X$.

Premetric spaces and their non-expanding maps form a category $\PMetr$ that contains the category $\Metr$ of metric spaces as a full subcategory.
\smallskip

Many notions related to metric spaces still can be defined for premetric spaces. In particular, for any point $x$ of a premetric space $X$ we can define the ball of radius $r$ centered at $x$ by the familiar formula:
$$B_X(x,r)=\{y\in X:d_X(x,y)<r\}.$$

Also we can define a subset $U$ of a premetric space $X$ to be {\em open} if for each point $x\in U$ there is $r>0$ with $B_X(x,r)\subset U$.  Open subsets of a premetric space $X$ form a topology called the {\em premetric topology}, see \cite[\S2.4]{AF}. Saying about topological properties of premetric spaces we shall always refer to this topology.

The following proposition can be immediately derived from the definition of the premetric topology.

\begin{proposition} Each non-expanding map between premetric spaces is continuous.
\end{proposition}

Each subset $A$ of a premetric space $X$ carries the induced premetric $d_A=d_X|A\times A$. Then the identity inclusion $\id:A\to X$ is non-expanding and hence is continuous with respect to the premetric topologies. However this inclusion is not necessarily a topological embedding.

\begin{example}[Arens' space]\label{arens} Consider the set 
$$S_2=\{(0,0)\}\cup\{(\tfrac1n,0):n\in\IN\}\cup\{(\tfrac1n,\tfrac1{nm}):n,m\in\IN\}\subset\IR^2$$ endowed the premetric 
$$d(x,y)=\begin{cases}0&\mbox{if $x=y$,}\\
\frac1n&\mbox{if $x=(0,0)$ and $y=(\frac1n,0)$,}\\
\frac1{nm}&\mbox{if $x=(\frac1n,0)$ and $y=(\frac1n,\frac1{nm})$,}\\
1&\mbox{otherwise.}
\end{cases}
$$ The topology on $S_2$ generated by this premetric coincides with the largest topology that induces the Euclidean topology on the convergent sequences $$\{(0,0)\}\cup\{(\tfrac1n,0):n\in\IN\}\mbox{  and }\{(\tfrac1n,0)\}\cup\{(\tfrac1n,\tfrac1{nm}):m\in\IN\},\;n\in\IN.$$

It follows that $(0,0)$ is a non-isolated point of the subset $A=\{(0,0)\}\cup\{(\frac1n,\frac1{nm}):n,m\in\IN\}\subset S_2$. On the other hand, the induced premetric $d_A=d|A\times A$ on $A$ is $\{0,1\}$-valued and generates the discrete topology on $A$. This means that the indentity inclusion $\id:A\to X$ is not a topological embedding.
\end{example}

This example suggests the following 

\begin{definition} A premetric space $X$ is called {\em hereditary} if for any subset $A\subset X$ the subspace topology on $A$ coincides with the topology generated by the premetric $d_A=d_X|A^2$ induced from $X$.
\end{definition}

In other to characterize hereditary premetric spaces let us introduce another

\begin{definition} A premetric $d$ on $X$ is called {\em basic at a point} $x\in X$ if the family of balls $\{B_X(x,r)\}_{r>0}$ is a neighborhood basis at $x$. This is equivalent to saying that $x$ is an interior point of each ball $B(x,r)$, $r>0$.

A premetric space $X$ is called {\em basic} if its premetric $d_X$ is basic at each point $x\in X$. In this case the premetric $d_X$ also is called {\em basic}.
\end{definition}

\begin{theorem}\label{t3.6} A premetric space is hereditary if and only if it is basic.
\end{theorem}

\begin{proof} Assume that a premetric space $X$ is not basic. This means that some point $x\in X$ fails to be an interior point of some ball $B_X(x,r)$, $r>0$. It follows that $x$ is a non-isolated point of the subspace $A=\{x\}\cup (X\setminus B_X(x,r))$ of $X$. On the other hand, this point is isolated in the topology generated by the restriction $d_A=d_X|A\times A$ of the premetric $d_X$ on $A$.
This means that the premetric space $X$ is not hereditary.
\smallskip

Now assume that a premetric space $X$ is basic.
Given any subset $A\subset X$, endow it with the induced premetric $d_A=d_X|A^2$ and consider the identity inclusion $\id_A:A\to X$. The hereditary property of $X$ will follow as soon as we check that the map $\id_A$ is a topological embedding. The continuity of $\id_A$ follows from the non-expanding property of $\id_A$. To show the continuity of the inverse map $\id_A^{-1}:\id_A(A)\to A$, take any point $a\in \id_A(A)$ and a neighborhood $U_a\subset A$ of $a$. Next, find a positive $r>0$ such that $B_A(a,r)\subset U_a$. Since the space $X$ is basic, the ball $B_X(a,r)\subset X$ contains some open neighborhood $V\subset X$ of $x$. Since $\id_A^{-1}(V)\subset \id_A^{-1}(B_X(a,r))=B_A(a,r)\subset U_a$, we see that the map $\id_A^{-1}:\id_A(A)\to X$ is continuous.
\end{proof}

Premetric spaces are tightly connected with weakly first-countable spaces introduced by A.V.~Arhan\-gel'skii in \cite{Arh}. 
We recall that a topological space $X$ is {\em weakly first-countable} if to each point $x\in X$ one can assign a decreasing sequence
$(B_n(x))_{n\in\w}$ of subsets of $X$ that contain $x$ so that a subset
$U\subset X$ is open if and only if for each $x\in U$ there is $n\in\w$ with $B_n(x)\subset U$. 

It is clear that each first-countable space is weakly first-countable. The Arens' space defined in Example~\ref{arens} is weakly first-countable but not first countable.

The following proposition shows that premetric spaces relate to weakly first-countable spaces in the same way as metric spaces relate to metrizable topological spaces.

\begin{proposition}\label{p3.7} A topological space $X$ is weakly first-countable (resp. first-countable) if and only if the topology of $X$ is generated by a premetric (resp. by a basic premetric).
\end{proposition}

\begin{proof} If the topology of $X$ is generated by a (basic) premetric $d:X\times X\to[0,\infty)$, then the family of balls $\big(B_X(x,\frac1n)\big)_{n\in\IN}$, $x\in X$, witnesses that $X$ is weakly first-countable (resp. first-countable). 

Now assume conversely that $X$ is weakly first-countable  and for every $x\in X$ fix a decreasing sequence of sets  $\big\{B_n(x)\big\}_{n\in\w}$ witnessing that $X$ is weakly first countable. If $X$ is first-countable, then we can additionally require that each set $B_n(x)$ is a neighborhood of $x$.

Those sequences determine the premetric $d:X\times X\to[0,\infty)$ defined by
$$d(x,y)=\inf\{2^{-n}:y\in B_n(x)\}$$ for $(x,y)\in X\times X$.
This premetric is basic if and only if each $B_n(x)$ is a neighborhood of $x$ in $X$.

Finally, observe the topology of $X$ is generated by the premetric $d$. 
\end{proof}

Our next aim is to characterize premetric spaces whose topology is first-countable.
It is clear that each basic premetric space is first-countable.
However the converse is not true.

\begin{example} On the set of integers $\IZ$ consider the premetric $$d(n,m)=\begin{cases}2^{-m}&\mbox{ if $n=0$ and $m>0$}\\
0&\mbox{otherwise}.
\end{cases}$$The topology on $\IZ$ generated by this premetric is anti-discrete and hence first-countable. However, the premetric space $(\IZ,d)$ is not basic.
\end{example}

On the other hand, we shall show that for sequentially Hausdorff premetric spaces the first-countability of the premetric topology is equivalent to the basic property of the premetric.

A topological space $X$ is called {\em sequentially Hausdorff\/} if any convergent sequence $\{x_n\}_{n\in\w}\subset X$ has a unique limit in $X$. Each Hausdorff space is sequentially Hausdorff. The converse is not true: any uncountable space $X$ endowed the cocountable topology $$\tau=\{\emptyset\}\cup\{U\subset X:|X\setminus U|\le\aleph_0\}$$is sequentially Hausdorff but not Hausdorff.

We shall characterize sequentially Hausdorff premetric spaces as 2-separating premetric spaces.

\begin{definition} A premetric space $X$ is called 
\begin{itemize}
\item {\em 1-separating} if for any distinct points $x,y\in X$ there is $r>0$ such that $y\notin B_X(x,r)$;
\item {\em 2-separating} if for any distinct points $x,y\in X$ there is $r>0$ such that $B_X(x,r)\cap B_X(y,r)=\emptyset$.
\end{itemize}
\end{definition}

The following characterization of 1-separating premetric spaces is immediate.

\begin{proposition}  A premetric space $X$ is 1-separated if and only if it is a topological $T_1$-space.
\end{proposition}

It is clear that a premetric space is 2-separating if its topology is Hausdorff. The converse is not true:

\begin{example} Take any non-metrizable first-countable compact Hausdorff space $K$. By Katetov Theorem \cite{Kat} (see also \cite[4.5.15]{En}), the cube $K^3$ is not hereditarily normal and hence contains a non-normal subspace $Y\subset K^3$. Consequently, $Y$ contains two closed disjoint sets $A,B\subset Y$ that have no disjoint neighborhoods in $Y$. It follows that the closures $\bar A$ and $\bar B$ of those sets in $K^3$ contain a common point $c\in \bar A\cap\bar B$. Fix a decreasing neighborhood base $(B_n(x))_{n\in\IN}$ at each point $x\in K^3$ such that $B_1(x)=K^3$.

It is easy to see that the space $X=Y\cup\{A,B\}$ endowed with the premetric $d$ defined by
$$d(x,y)=\begin{cases}\inf\{\tfrac1n:y\in B_n(x)\}&\mbox{if $x,y\in Y$},\\
\inf\{\tfrac1n:y\in B_n(c)\}&\mbox{if $x=A$ and $y\in A$},\\  
\inf\{\tfrac1n:y\in B_n(c)\}&\mbox{if $x=B$ and $y\in B$},\\  
1&\mbox{otherwise}
\end{cases}
$$
is 2-separating but not Hausdorff.
\end{example}

On the other hand, we have the following characterization:

\begin{theorem} A premetric space  is 2-separating if and only if it is sequentially Hausdorff.
\end{theorem}

This theorem can be easily derived from the following useful characterization of convergence in 2-separating premetric spaces.

\begin{proposition}\label{p3.13} A sequence $(x_n)_{n\in\w}$ in a 2-separating premetric space $X$ converges to a point $x\in X$ if and only if $\lim\limits_{n\to\infty}d_X(x,x_n)=0$.
\end{proposition}

\begin{proof} The ``if'' part is trivial. To prove the ``only if'' part, assume that a  sequence $(x_n)_{n\in\w}$ converges to $x$ but $\lim\limits_{n\to\infty}d_X(x,x_n)\ne0$. This means that for some $r>0$ the set $N=\{n\in\w:x_n\notin B_X(x,r)\}$ is infinite. The convergence of the sequence $(x_n)$ to $x$ implies that the set $U=X\setminus\{x_n\}_{n\in N}$ is not open in $X$.

We shall derive a contradiction showing that for every $y\in U$ there is $\e>0$ with $B_X(y,\e)\subset U$. Assuming that no such $\e$ exists, we conclude that for some infinite subset $M\subset N$ we get $d_X(y,x_m)\to 0$ as $M\ni m\to\infty$.
In this case we show that the set $V=X\setminus (\{y\}\cup\{x_m\}_{m\in M})$ is open in $X$. Indeed given any point $z\in V$, we can use the 2-separating property of $X$ and  find $\e>0$ so small that $B_X(z,\e)\cap B(y,\e)=\emptyset$. Since the set $F=\{m\in M:x_m\notin B_X(y,\e)$ is finite, we can find a positive $\delta\le\e$ such that $B_X(z,\delta)\cap F=\emptyset$. It follows that $B(z,\delta)\subset V$, witnessing that the set $V$ is open.

Taking into account that $\{x_m\}_{m\in M}\cap B_X(x,r)=\emptyset$ and $\lim\limits_{M\ni m\to\infty}d_X(y,x_m)\to 0$, we conclude that $x\ne y$ and hence $V$ is an open neighborhood of $x$. Since $M\subset\{n\in\w:x_n\notin V\}$ is infinite, we conclude that the sequence $(x_n)_{n\in\w}$ does not converge to $x$, which is a desired contradiction.
\end{proof}

Now we are ready to characterize first-countable spaces among 2-separated premetric spaces. This will be done with help of the Fr\'echet-Urysohn property.

We recall that a topological space $X$ is
\begin{itemize}
\item {\em Fr\'echet-Urysohn} if for each $A\subset X$ and $x\in \cl_X(A)$ there is a sequence $\{a_n\}_{n\in\w}\subset A$ that converges to $x$;
\item {\em sequential} if for each non-closed subset $A\subset X$ there is a sequence $\{a_n\}_{n\in\w}\subset A$ that converges to a point $x\in X\setminus A$.
\end{itemize}

By \cite{Arh}, each weakly first-countable space is sequential. Combining this result with Proposition~\ref{p3.7} we get

\begin{proposition} Each premetric space is sequential.
\end{proposition}

\begin{proof} We shall give a direct proof of this important fact for the convenience of the reader. We need to show that a subset $A\subset X$ of a premetric space is closed if for each convergent sequence $\{a_n\}_{n\in\w}\subset A$ we get $\lim\limits_{n\to\infty}x_n\in A$. Assuming that such a set $A$ is not closed in the premetric topology of $X$, we can find a point $x\in X\setminus A$ such that $B_X(x,\frac1n)\not\subset X\setminus A$ for any $n\in\IN$. This allows us to choose a point $a_n\in B_X(x,\frac1n)\cap A$. The definition of the premetric topology on $X$ guarantees that the sequence $\{a_n\}_{n\in\IN}\subset A$ converges to the points $x\in X\setminus A$, which contradicts the choice of the set $A$.
\end{proof}

The following theorem characterizes first-countable premetric spaces among 2-separating premetric spaces.

\begin{proposition} For a 2-separating premetric space $X$ the following conditions are equivalent:
\begin{enumerate}
\item $X$ is basic;
\item $X$ is hereditary.
\item $X$ is first-countable;
\item $X$ is Fr\'echet-Urysohn;
\end{enumerate}
If the equivalent conditions (1)-(4) hold, then $X$ is Hausdorff.
\end{proposition}

\begin{proof} The equivalene $(1)\Leftrightarrow(2)$ has been proved in Theorem~\ref{t3.6}. The implications $(1)\Ra(3)\Ra(4)$ are trivial. 

$(4)\Ra(1)$ Assume that a 2-separated premetric space $X$ is Fr\'echet-Urysohn. Assuming that $X$ is not basic, we would find a point $x\in X$ and a radius $r>0$ such that $x$ is not an interior point of the ball $B_X(x,r)$. Since $X$ is Fr\'echet-Urysohn, there is a sequence $\{x_n\}_{n\in\w}\subset X\setminus B_X(x,r)$ that converges to $x$. But this contradicts Proposition~\ref{p3.13}.
\end{proof} 

\section{The sequential decomposition of a topological space}\label{dec}

In this section we shall describe a functor $\Dec:\Top\to\PMetr$ assigning to each topological space $X$ a premetric space $\Dec X$ called the {\em sequential decomposition} of $X$. This space is defined as follows.

Let $S_0=\{0\}\cup\{\frac1n:n\in\IN\}$ be a standard convergent sequence on the real line.
Then each convergent sequence $(x_n)_{n\in\IN}$ in a topological space $X$ can be identified with the continuous map $f:S_0\to X$ that maps $\frac1n$ onto $x_n$ and $0$ onto the limit $\lim\limits_{n\to\infty}x_n$.

Therefore, the set $C(S_0,X)$ of all continuous functions from $S_0$ to $X$ can be identified with the family of all convergent sequences in $X$.

The premetric space $\Dec X$ is defined as the set $$\Dec X=C(S_0,X)\times S_0$$ endowed with the premetric 
$$d_{\Dec X}\big((f,t),(g,s)\big)=\begin{cases}0&\mbox{if $f(t)=g(s)$},\\
|t-s|&\mbox{if $f(t)\ne g(s)$ but $f=g$},\\
1&\mbox{otherwise}.
\end{cases}
$$

Any continuous map $f:X\to Y$ between topological spaces induces a non-expanding map $$\Dec f:\Dec X\to\Dec Y,\quad\Dec f:(\xi,s)\mapsto (f\circ\xi,s).$$ 

In such a way we have defined a functor
$$\Dec:\Top\to\PMetr$$ from the category $\Top$ of topological spaces and their continuous maps to the category $\PMetr$ of premetric spaces and their non-expanding maps.
\smallskip

For every topological space $X$ the spaces $\Dec X$ and $X$ are linked by the {\em calculation map} $$c_X:\Dec X\to X,\quad c_X:(f,x)\mapsto f(x).$$

This map is natural in the sense that for any continuous map $f:X\to Y$ between topological spaces the following diagram is commutative:
$$\begin{CD}
\Dec X@>{\Dec f}>>\Dec Y\\
@V{c_X}VV@VV{c_Y}V\\
X@>>f>Y
\end{CD}
$$

The calculation map will be used for describing the topology of the space $\Dec X$. Namely, we shall show that a subset $U\subset\Dec X$ is open if and only if $U=c_X^{-1}(V)$ for some sequentially open subset $V\subset U$.

A subset $V\subset X$ is called {\em sequentially open} if for each sequence $(x_n)_{n\in\w}$ in $X$ that converges to a point $x\in V$ there is $n\in\w$ such that $x_m\in V$ for all $m\ge n$. 

\begin{lemma}\label{l4.1} A subset $U\subset\Dec X$ is open if and only if $U=c_X^{-1}(V)$ for some sequentially open subset $V\subset U$.
\end{lemma}

\begin{proof} First assume that $U=c_X^{-1}(V)$ for some sequentially open subset $V\subset X$. In order to show that $U$ is open in $\Dec X$ we should find for any point $a\in U\subset\Dec X$ a radius $r>0$ such that $B_{\Dec X}(a,r)\subset U$. 

Being an element of $\Dec X$, the point $a$ is of the form $a=(f,t)$ for some continuous map $f:S_0\to X$ and some $t\in S_0$. If $t\ne0$, then we can find $r>0$ such that $(t-r,t+r)\cap S_0=\{t\}$. For such $r$ the ball $B_{\Dec X}(a,r)$ coincides with the set $\{(g,s)\in\Dec X:g(s)=f(t)\}$ and then $c(B_{\Dec X}(a,r))=\{f(t)\}=\{c(a)\}\subset V$ and hence $B_{\Dec X}(a,r)\subset U$.
If $t=0$, then the sequence $(f(\frac1n))_{n=1}^\infty$ converges to $f(0)=c(a)\in V$. As $V$ is sequentially open, there is $m\in\IN$ such that $f(\frac1n)\in V$ for all $n\ge m$. Then for the radius $r=\frac1m$ we get $B_{\Dec X}(a,r)\subset c_X^{-1}(V)=U$.

Now assume that a set $U\subset\Dec X$ is open. First we show that $U=c_X^{-1}(c_X(U))$.  Given any $a=(f,t)\in U$ find $r>0$ with $B_{\Dec X}(a,r)\subset U$ and observe that $$c^{-1}(c(a))=\{(g,s)\in\Dec X:g(s)=f(t)\}\subset\{b\in\Dec X:d_{\Dec X}(a,b)=0\}\subset B_{\Dec X}(a,r)\subset U$$witnessing that $U=c_X^{-1}(c_X(U))$. It remains to check that the set $c_X(U)$ is sequentially open in $X$. Fix a sequence $(x_n)_{n\in\IN}$ in $X$ that converges to a point $x_0\in c_X(U)$. Then the map $f:S_0\to X$ defined by $f(0)=x_0$ and $f(\frac1n)=x_n$ for $n\in\IN$ is continuous.
Consider the point $a=(f,0)\in\Dec X$ and observe that $c_X(a)=x_0\in c_X(U)$. Since $U$ is open, for  the point $a\in c_X^{-1}(x_0)\subset U$ there is a radius $r>0$ such that $B_{\Dec X}(a,r)\subset U$. Then for every $n>\frac1r$ we get $(f,\frac1n)\in B_{\Dec X}(a,r)\subset U$ and hence $x_n=c_X(f,\frac1n)\in c_X(U)$, witnessing that the set $c_X(U)$ is sequentially open in $X$.
\end{proof}

Sequentially open subsets of $X$ form a topology on $X$. 
The set $X$ endowed with this topology is called the {\em sequential coreflexion} of $X$ and is denoted by $\mathsf s X$.

Since each open subset of $X$ is sequentially open, the identity map $\id:\seq X\to X$ is continuous. This map is a homeomorphism if and only if the space $X$ is sequential. 

Lemma~\ref{l4.1} implies the following important result:

\begin{corollary}\label{c4.2} For every topological space $X$ the map $c_X:\Dec X\to \seq X$ is  surjective, continuous, monotone, pseudo-open, and herediarily quotient. 
\end{corollary}

We recall that a surjective map $f:X\to Y$ between topological spaces is
\begin{itemize}
\item {\em monotone} if $f^{-1}(y)$ is connected for every $y\in Y$;
\item {\em quotient} if a set $U\subset Y$ is open if and only if its preimage $f^{-1}(U)$ is open in $X$;
\item {\em hereditarily quotient} if for every subset $A\subset Y$ the restriction $f|f^{-1}(A):f^{-1}(A)\to A$ is a quotient map;
\item {\em pseudo-open} if for every $y\in Y$ and an open set $U\subset X$ containing $f^{-1}(y)$ the point $y$ lies in the interior of $f(U)$ in $Y$.
\end{itemize}

We shall often use the following characterization of hereditarily quoteint maps, due to Arhangel'skii \cite{Ar2}, see also \cite[2.4.F(a)]{En}.

\begin{theorem}\label{hq=po} A map between topological spaces is hereditarily quotient if and only if it is continuous and pseudo-open.
\end{theorem}

It is known that each quotient map onto a sequentially Hausdorff Fr\'echet-Urysohn space is hereditarily quotient and pseudo-open, see \cite{Ar2} or \cite[2.4.F(c)]{En}.

Quotient maps are important for us because of the following known fact, see \cite[6.1.28]{En}.

\begin{lemma}\label{l4.4} Let $f:X\to Y$ be a monotone quotient map. The space $X$ is connected if and only if $Y$ is connected.
\end{lemma}

This lemma combined with Corollary~\ref{c4.2} implies:

\begin{corollary} The sequence decomposition $\Dec X$ of a topological space $X$ is connected if and only if the sequential coreflexion $\seq X$ of $X$ is connected. In particular, $\Dec X$ is connected for every connected sequential topological space $X$.
\end{corollary}

\section{The complete oriented graph over a set}\label{graph}

The complete oriented graph $\Gamma X$ over a set $X$ is the set
$$\Gamma X=X\cup\{(x,y,t)\in X\times X\times (0,1):x\ne y\}$$ endowed with a special (path) metric. 
The points $x\in X\subset\Gamma X$ are called {\em vertices} of $\Gamma X$ and the sets 
$$(x,y)=\{(x,y,t):t\in(0,1)\}\mbox{ \ \ and \ \ }[x,y]=\{x,y\}\cup(x,y)$$ are called {\em oriented edges} connecting distinct vertices $x,y\in X$. It will be convenient to consider also degenerated edges $(x,x)=\emptyset$ and $[x,x]=\{x\}$ for $x\in X$.

Therefore, $$\Gamma X=\bigcup_{x,y\in X}[x,y]=X\cup\bigcup_{x\ne y}(x,y).$$
Observe that the oriented edges $[x,y]$ and $[y,x]$ meet only by their endpoints: $[x,y]\cap[y,x]=\{x,y\}$. 
\smallskip

The graph $\Gamma X$ is the image of the product $X\times X\times [0,1]$ under the map $\la\cdot\ra:X\times X\times [0,1]\to\Gamma X$ defined by
$$\la\cdot\ra:(x,y,t)\mapsto\la x,y,t\ra=\begin{cases}(x,y,t)&\mbox{if $x\ne y$ and $t\in(0,1)$}\\
x&\mbox{if $x=y$ or $t=0$}\\
y&\mbox{if $x=y$ or $t=1$.}
\end{cases}
$$ In particular, $\la x,y,0\ra=\la x,x,t\ra=\la y,x,1\ra=x$ for every $x,y\in X$ and $t\in[0,1]$.
\smallskip

Now, we define a metric $d_{\Gamma X}$ on $\Gamma X$ such that each edge $[x,y]\subset\Gamma X$, $x\ne y$, is isometric to the unit interval $[0,1]$. For this we put 
$$\tilde d_{\Gamma X}(\la x,y,t\ra,\la x,y,s\ra)=|t-s|$$ for any distinct vertices $x,y\in X$ and $t,s\in[0,1]$. Next, extend $\tilde d_{\Gamma X}$ to a metric $d_{\Gamma X}$ on $\Gamma X$ letting $d_{\Gamma X}(a,b)$ be the minimal value of the sum $\sum_{i=1}^n\tilde d_{\Gamma X}(a_{i-1},a_i)$ where $a=a_0,a_1,\dots,a_n=b$ and for any $i\le n$ the points $a_{i-1},a_i$ lie on some common edge $[x,y]\subset\Gamma X$.

It is easy to check that the so-defined metric $d_{\Gamma X}$ on $\Gamma X$ has the following properties:

\begin{proposition}\label{p5.1}\begin{enumerate}
\item The metric $d_{\Gamma X}$ is complete and $\mathrm{diam}\,\Gamma X\le 2$.
\item Each edge $[x,y]$, $x\ne y$, is isometric to $[0,1]$.
\item Each edge $(x,y)$, $x\ne y$, is open in $\Gamma X$.
\item Any two points $a,b\in\Gamma X$ lying on closed disjoint edges are on the distance $d_{\Gamma X}(a,b)\ge 1$.
\item If $d_{\Gamma X}(a,b)<1$ for some $a,b\in\Gamma X$, then either $a,b$ lie on the same edge or else there is a vertex $x\in X$ such that $d_{\Gamma X}(a,b)=d_{\Gamma X}(a,x)+d_{\Gamma X}(x,b)$.
\end{enumerate}
\end{proposition}

The construction of the complete oriented graph determines a functor $$\Gamma:\Set\to\Metr$$ from the category $\Set$ of sets and their maps to the category $\Metr$ of metric spaces and their non-expanding maps.

In fact, each map $f:X\to Y$ between sets determines a map $\Gamma f:\Gamma X\to\Gamma Y$ defined by
$$\Gamma f:\la x,z,t\ra\mapsto\la f(x),f(z),t\ra.$$

It follows from the definition of the metrics on the graphs $\Gamma X$ and $\Gamma Y$ that the map $\Gamma f:\Gamma X\to\Gamma Y$ is non-expanding in the sense that $d_{\Gamma Y}(\Gamma f(a),\Gamma f(b))\le d_{\Gamma X}(a,b)$ for any $a,b\in\Gamma X$. Also it is clear that the identity map $\mathrm{id}:X\to X$ induces the identity map $\Gamma\,\mathrm{id}:\Gamma X\to\Gamma X$ and for any maps $f:X\to Y$ and $g:Y\to Z$ the composition $\Gamma g\circ\Gamma f$ equals to map $\Gamma(g\circ f)$. This means that $\Gamma:\Set\to\Metr$ is a functor.

\begin{proposition}\label{p5.2} For every injective map $f:X\to Y$ between sets the induced map $\Gamma f:\Gamma X\to\Gamma Y$ is an isometric embedding.
\end{proposition}

\begin{proof} The map $f$, being injective, admits a left inverse map $g:Y\to X$ such that $g\circ f=\id_X$. By the functoriality of $\Gamma$, the maps $\Gamma f:\Gamma X\to\Gamma Y$ and $\Gamma g:\Gamma Y\to\Gamma X$ are non-expanding and their composition $\Gamma g\circ\Gamma f=\Gamma (g\circ f)=\Gamma\,\id_X=\id_{\Gamma X}$. Then for every $a,b\in\Gamma X$ the non-expanding property of the maps $\Gamma f$ and $\Gamma g$ yields:
$$d_{\Gamma X}(a,b)=d_{\Gamma X}(\Gamma g(\Gamma f(a)),\Gamma g(\Gamma f(b))\le d_{\Gamma Y}(\Gamma f(a),\Gamma f(b))\le d_{\Gamma X}(a,b),$$
which implies that $d_{\Gamma Y}(\Gamma f(a),\Gamma f(b))=d_{\Gamma X}(a,b)$. This witnesses that $\Gamma f:\Gamma X\to\Gamma Y$ is an isometric embedding.
\end{proof}

\section{The cobweb construction over a premetric space}\label{cobweb}

In this section we define the cobweb functor $\Cob:\Dist\to\Metr$. 

By the {\em cobweb} of a premetric space $X$ we understand the following closed subspace $\Cob X$ of the complete oriented graph $\Gamma X$:
$$\Cob X=\{\la x,y,t\ra\in\Gamma X:t\le 1-\bar d_X(y,x)\}\big\}\subset\Gamma X,$$
where $\bar d_X=\min\{1,d_X\}$ and $d_X$ is the premetric of $X$.

Let us explain this construction in more details. For any distinct points $x,y\in X$ consider the unique point $x_y$ on the edge $[x,y]\subset\Gamma X$ that lies on the distance $d_{\Gamma X}(x_y,y)=\bar d_X(y,x)=\min\{1,d_X(y,x)\}$ from the vertex $y$. It follows that the set $$[x,x_y]=\{\la x,y,t\ra\in\Gamma X:t\le d_{\Gamma X}(x,x_y)=1-\bar d_X(y,x)\}$$ is a subarc of length $1-\bar d_X(y,x)$ in $[x,y]\subset\Gamma X$. If $d_X(y,x)\ge1$, then the subarc $[x,x_y]$ degenerates to the singleton $\{x\}$.

The union $$S_x=\bigcup_{y\ne x}[x,x_y]\setminus\{y\}$$ will be called the {\em spider centered at} a point $x\in X\subset \Cob X\subset\Gamma X$. Each spider $S_x$ is an arcwise connected subspace of $\Gamma X$.
Therefore $$\Cob X=\bigcup_{x\ne y}[x,x_y]=\bigcup_{x\in X}S_x$$ is the union of all such spiders.

The cobweb $\Cob X$ contains $X$ as the discrete set of vertices. Consequently,
$$|X|\le\dens (\Cob X)\le|\Cob X\,|\le\mathfrak c\cdot|X|.$$
The reader should not be confused by the fact that the space $X$ that has the premetric topology will be also considered as the discrete subspace of the cobweb $\Cob X$ over $X$.

Since $\Cob X$ can be obtained from $\Gamma X$ by deleting selected open arcs from the edges of $\Gamma X$, the subspace $\Cob X$ is closed in $\Gamma X$ and hence is complete with respect to the metric $d_{\Cob X}=d_{\Gamma X}|\Cob X\times\Cob X$ induced from the graph $\Gamma X$.
\smallskip

Next, we show that the construction of the cobweb space determines a functor $$\Cob:\Dist\to\Metr$$ from the category $\PMetr$ of premetric spaces and their non-expanding maps to the category $\Metr$ of metric spaces.

Given any non-expanding map $f:X\to Y$ between premetric spaces, consider the induced non-expanding map $\Gamma f:\Gamma X\to\Gamma Y$ between the complete oriented graphs. 

\begin{lemma}\label{l6.1} $\Gamma f(\Cob X)\subset\Cob Y$. 
\end{lemma}

\begin{proof}  Given any point $a\in\Cob X\subset\Gamma X$, find vertices $x,y\in X$ such that $a\in[x,x_y]$. It follows that $d_{\Gamma X}(x,a)\le d_{\Gamma X}(x,x_y)=1-\bar d_X(y,x)$ where $\bar d_X=\min\{1,d_X\}$. Now consider the image $\Gamma f(a)\subset\Gamma Y$. If $f(a)\in Y$, then we are done because $Y\subset\Cob Y\subset\Gamma Y$. So, we assume that $f(a)\in\Gamma Y\setminus Y$. In this case $f(a)=(f(x),f(y),t)$ where $a=(x,y,t)$ and $t\le 1-\bar d_X(y,x)$. Since the map $f$ is non-expanding, $\bar d_Y(f(y),f(x))\le \bar d_X(y,x)$ and consequently, $t\le 1-\bar d_X(y,x)\le 1-\bar d_Y(f(y),f(x))$. Now the definition of $\Cob Y$ guarantees that $f(a)=(f(x),f(y),t)\in [f(x),f(x)_{f(y)}]\subset\Cob Y$.
\end{proof}

Lemma~\ref{l6.1} allows us to define a map $\Cob f:\Cob X\to\Cob Y$ as the restriction $\Cob f=\Gamma f|\Cob X$. In such a way we have proved 

\begin{theorem} $\Cob:\Dist\to\Metr$ is a functor.
\end{theorem}

Proposition~\ref{p5.2} implies that this functor preserves isometric embeddings. 

\begin{proposition} For any injective non-expanding map $f:X\to Y$ between premetric spaces the induced map $\Cob f:\Cob X\to\Cob Y$ is an isometric embedding.
\end{proposition}

\section{The compression map}\label{compression}

In this section, given a premetric space $X$ we construct an important map $\pi_X:\Cob X\to X$ called the {\em compression map}. It is defined by the formula:
$$\pi_X(a)=\begin{cases}
a&\mbox{if $a\in X\subset\Cob X$};\\
x&\mbox{if $a=(x,y,t)\in\Cob X\setminus X$}.
\end{cases}
$$
Observe that for every $x\in X$ the preimage $$\pi_X^{-1}(x)=S_x=\bigcup_{X\ni y\ne x}[x,x_y]\setminus\{y\}$$ coincides with the spider $S_x$ centered at $x$.

The compression map $\pi_X$ is natural in the sense that for any non-expanding map $f:X\to Y$ between premetric spaces the following diagram commutes:
$$\begin{CD}
\Cob X@>{\pi_X}>> X\\
@V{\Cob f}VV@ VV{f}V\\
\Cob Y@>>{\pi_Y}> Y
\end{CD}
$$

In the following proposition we collect some simple properties of the compression map.

\begin{proposition}\label{p7.1}  Let $X$ be a premetric space. Then:\begin{enumerate}
\item $\pi_X(x)=x$ for all $x\in X$;
\item $\pi_X:\Cob X\to X$ is a monotone surjection;
\item $\pi_X$ is locally constant on the open subset $\Cob X\setminus X$ of $\Cob X$;
\item for any subset $A\subset\Cob X$ we get $|\pi_X(A)|\le\dens(A)$.
\end{enumerate}
\end{proposition}

\begin{proof} 1,3. The first and third items follow immediately from the definition of $\pi_X$.
\smallskip

2. Since each spider $S_x=\pi_X^{-1}(x)$, $x\in X$, is not empty and (arcwise) connected, the map $\pi_X$ is a monotone surjection.
\smallskip

4. Let $A\subset \Cob X$. If $A$ is finite, then 
$|\pi_X(A)|\le|A|=\dens(A)$. So, we assume that $A$ is infinite.

Since $X$ is discrete in $\Cob X$, the intersection $A\cap X$ has cardinality $|A\cap X|\le \dens(A)$. Since $\pi_X$ is locally constant on the set $\Cob X\setminus X$, the image $\pi_X(A\setminus X)$ has cardinality $|\pi_X(A\setminus X)|\le\dens(A)$. Combining those two facts, we get 
$$|\pi_X(A)|\le|\pi_X(A\cap X)|+|\pi_X(A\setminus X)|\le \dens(A)+\dens(A)=\dens(A).$$
\end{proof}

Next, we establish some metric properties of the compression map. 

\begin{proposition}\label{p7.2} Let $X$ be a premetric space, $d_X$ be the premetric of $X$ and $\bar d_X=\min\{1,d_X\}$. Then 
\begin{enumerate}
\item $\bar d_X(x,\pi_X(a))\le d_{\Cob X}(x,a)$ for every $x\in X$ and $a\in\Cob X$;
\item $\pi(B_{\Cob X}(x,r))=B_X(x,r)$ for any $x\in X$ and $r\in(0,1]$;
\item $d_{\Cob X}(a,b)\ge\inf_{x\in X}\bar d_X(x,\pi_X(a))+\bar d_X(x,\pi_X(b))$ for any $a,b\in\Cob X$;
\item the compression map $\pi_X:\Cob X\to (X)$ is non-expanding if and only if $\bar d_X$ is a pseudometric.
\end{enumerate}
\end{proposition}

\begin{proof} 1. Fix any $x\in X$ and $a\in \Cob X$. We need to check that  $\bar d_X(x,y)\le d_{\Cob X}(x,a)$ where $y=\pi_X(a)$. This inequality is trivial if $d_{\Cob X}(x,a)\ge 1$. So, we assume that $d_{\Cob X}(x,a)<1$.

Since $a\in\pi^{-1}(y)=S_y=\bigcup_{z\ne y}[y,y_z]\setminus\{z\}$, there is a points $z\in X\setminus \{y\}$ such that $a\in [y,y_z]$. By Proposition~\ref{p5.1}(4), the inequality $d_{\Cob X}(x,a)<1$ implies that $x\in\{y,z\}$. If $x=y$, then  we are done because $\bar d_X(x,y)=0\le d_{\Cob X}(x,a)$.
If $x=z$, then $\bar d_X(x,y)=\bar d_X(z,y)=d_{\Cob X}(z,y_z)\le d_{\Cob X}(z,a)=d_{\Cob X}(x,a)$. 
\smallskip

2. Fix any $x\in X$ and a positive $r\le1$. The preceding item implies that 
$\pi_X(B_{\Cob X}(x,r))\subset B_X(x,r)$. The reverse inclusion will follow as soon as, 
given any point $z\in B_X(x,r)$, we find a point $y\in \pi^{-1}_X(z)=S_z$ with $d_{\Cob X}(x,y)<r$.
If $z=x$, then we can put $y=x$. So, assume that $z\ne x$. 

Consider the point $z_x\in[z,x]$ and observe that $d_{\Cob X}(x,z_x)=\bar d_X(x,z)\le d_X(x,z)<r$. If $z_x\ne x$, then we can put $y=z_x\in S_z$. If $z_x=x$, then $[z,z_x]=[z,x]$ and we can take any point $y\in [z,z_x)\subset S_z$ with $d_{\Cob X}(x,y)<r$.
\smallskip

3. Take any two points $a,b\in\Cob X$ and consider their images $y=\pi_X(a)$ and $z=\pi_X(b)$. If $y=z$, then
$$\inf_{x\in X}\bar d_X(x,y)+\bar d_X(x,z)=0\le d_{\Cob X}(a,b).$$
So, we can assume that $y\ne z$.

If $d_{\Cob X}(a,b)\ge 1$, then
$$d_{\Cob X}(a,b)\ge 1\ge \bar d_X(y,z)\ge\inf_{x\in X}\bar d_X(x,y)+\bar d_X(x,z)$$ and we are done.
So, we can assume that $d_{\Cob X}(a,b)<1$. In this case Proposition~\ref{p5.1}(5) guarantees that $d_{\Cob X}(a,b)=d_{\Cob X}(a,v)+d_{\Cob X}(v,b)$ for some vertex $v\in X\subset\Gamma X$. Applying the item (1) we conclude that 
$$
\begin{aligned}
d_{\Cob X}(a,b)&=d_{\Cob X}(v,a)+d_{\Cob X}(v,b)\ge \bar d_X(v,\pi_X(a))+\bar d_X(v,\pi_X(b))\ge\\
&\ge\inf_{x\in X}\bar d_X(x,\pi_X(a)+\bar d_X(x,\pi_X(b)).
\end{aligned}$$
\smallskip

4. If $\bar d_X$ is a pseudometric, then for any $a,b\in\Cob X$ the triangle inequality for the pseudometric $\bar d_X$ combined with the preceding item implies 
$$d_{\Cob X}(a,b)\ge\inf_{x\in X}\bar d_X(x,\pi_X(a))+\bar d_X(x,\pi_X(b))\ge d_X(\pi_X(a),\pi_X(b)),$$witnessing the non-expanding property of the compression map $\pi_X$.

Now assume that the compression map $\pi_X:\Cob X\to X$ is non-expanding. Given any three points $x,y,z\in X$ we shall prove that 
\begin{equation}\label{triangle}
\bar d_X(x,z)\le \bar d_X(y,x)+\bar d_X(y,z).
\end{equation} This inequality is trivial if $x=z$ or $x=y$. So, we assume that $x\ne z$ and $x\ne y$. Consider the points $x_y,z_y\in\Cob X$. If $x_y\ne y\ne z_y$, then
$\pi_X(x_y)=x$, $\pi_X(z_y)=z$ and then
$$
\begin{aligned}\bar d_X(x,z)&=\bar d_X(\pi_X(x_y),\pi_X(z_y))\le d_{\Cob X}(x_y,z_y)\le\\
&\le d_{\Cob X}(x_y,y)+d_{\Cob X}(z_y,y)=\bar d_X(y,x)+\bar d_X(y,z).
\end{aligned}
$$

If $x_y=y$ and $z_y\ne y$, then for every $\e>0$ we can find a point $c\in [x,x_y)=[x,y)$ such that $d_{\Cob X}(c,y)<\e$ and conclude that
$$
\begin{aligned}
\bar d_X(x,z)&=\bar d_X(\pi_X(c),\pi_X(z_y))\le d_{\Cob X}(c,z_y)\le\\
&\le d_{\Cob X}(c,y)+d_{\Cob X}(y,z_y)<\e+\bar d_X(y,z)\le \bar d_X(y,x)+\bar d_X(y,z)+\e.
\end{aligned}$$
Passing to the limit at $\e\to 0$, we get the desired inequality (\ref{triangle}). By a similar argument we can treat the cases when $z_y=y$.

For points $x$ and $y=z$ the inequality (\ref{triangle}) yields $\bar d_X(x,y)\le \bar d_X(y,x)$. This implies that the premetric $\bar d_X$ is symmetric and satisfies the triangle inequality, so is a pseudometric.
\end{proof}

Next, we establish some topological properties of the compression map. 
In the following theorem we endow the premetric space $X$ with the premetric topology.

\begin{theorem}\label{t7.3} For any premetric space $X$ 
\begin{enumerate}
\item the compression map $\pi_X:\Cob X\to X$ is a monotone quotient surjection;
\item $\pi_X:\Cob X\to X$ is hereditarily quotient if and only if the premetric space $X$ is basic;
\item the cobweb space $\Cob X$ is connected if and only if so is the space $X$.
\end{enumerate}
\end{theorem}

\begin{proof} 1. By Proposition~\ref{p7.1}(2), $\pi_X$ is a monotone surjection. To show that $\pi_X$ is continuous, take any open subset $U\subset X$. To check that its  preimage $\pi^{-1}_X(U)$ is open in $\Cob X$, fix any point $a\in\pi^{-1}_X(U)$. 
If $a\in\Cob X\setminus X$, then the map $\pi_X$ is locally constant and hence continuous at $a$. So, we assume that $a\in X$. In this case the definition of the premetric topology on $X$ yields a radius  $r\in(0,1)$ such that $B_X(\pi_X(a),r)\subset U$. By Proposition~\ref{p7.2}(2), $\pi_X(B_{\Cob X}(a,r))\subset B_X(\pi_X(a),r)\subset U$ and hence $B_{\Cob X}(a,r)\subset \pi^{-1}_X(U)$, witnessing that the preimage $\pi^{-1}_X(U)$ is open in $\Cob X$.

 To show that $\pi_X$ is quotient, take any subset $A\subset X$ whose preimage $\pi^{-1}_X(A)$ is open in $\Cob X$. For every $x\in A$ we get $x\in\pi^{-1}_X(A)$ and hence $B_{\Cob X}(x,r)\subset\pi^{-1}_X(A)$ for some $r\in(0,1)$. It follows from Proposition~\ref{p7.2}(2) that $$B_X(x,r)=\pi_X(B_{\Cob X}(x,r))\subset\pi_X(\pi^{-1}_X(A))=A.$$Consequently, the set $A$ is open in the premetric topology of $X$.
\smallskip

2. Assuming that the premetric space $X$ is basic, we check that the compression map $\pi_X:\Cob X\to X$ is pseudo-open and hence hereditarily quotient. Fix any point $x\in X$ and any open set $U\subset\Cob X$ containing the fiber $\pi^{-1}_X(x)$. Since $x\in U$, there is a positive $r<1$ such that $B_{\Cob X}(x,r)\subset U$. Now Proposition~\ref{p7.2}(2) implies that $B_X(x,r)=\pi_X(B_{\Cob X}(x,r))\subset\pi_X(U)$ and hence $\pi_X(U)$ is a neighborhood of $x$ in $X$ (because the ball $B_X(x,r)$ is a neighborhood of $x$).

Now assume conversely that the compression map $\pi_X$ is hereditarily quotient and hence pseudo-open. We need to check that the premetric $d_X$ of $X$ is basic at each point $x\in X$. This will follow as soon as we check that for any positive $r<1$ the  ball $B_X(x,r)$ is a neighborhood of $x$ in $X$.
Observe that the set $U=S_x\cup B_{\Cob X}(x,r)$ is an open neighborhood of the fiber $S_x=\pi_X^{-1}(x)$ in $\Cob X$. Since $\pi_X$ is pseudo-open, the image $\pi_X(U)=\{x\}\cup \pi_X(B_{\Cob X}(x,r))=B_X(x,r)$ is a neighborhood of $x$ in $X$.
\smallskip

3. The third item follows from the item (1) and Lemma~\ref{l4.4}.
\end{proof}

Finally, we apply the compression map to studying separablewise components of the cobweb space $\Cob X$.

By the {\em separablewise component} of a point $x$ of a topological space $X$ we understand the union $C(x)$ of all separable connected subspaces of $X$ that contain the point $x$. It is standard to show that two separablewise components either are disjoint or else coincide. In a countably tight topological space all separablewise components are closed. We recall that a topological space $X$ is {\em countably tight} if for each subset $A\subset X$ and a point $x\in\bar A$ in its closure there is a countable subset $B\subset A$ such that $x\in\bar B$.

We shall say that a topological space $X$ {\em contains no countable connected subspaces} if each non-empty at most countable connected subset of $X$ is a singleton. For example, each regular space contains no non-trivial countable connected subspace.

\begin{proposition} If a premetric space $X$ contains no countable connected subspace, then the fibers of the compression map $\pi_X:\Cob X\to X$ coincide with the separablewise components of $\Cob X$ and also with the arcwise components of $\Cob X$. Consequently, $X$ can be identified with the space of separablewise (or arcwise) connected components of $\Cob X$.
\end{proposition}

\begin{proof} Assuming that $X$ contains no countable connected subspace, we shall show that each spider $S_x=\pi_X^{-1}(x)$, $x\in X$, coincides with the separablewise component $C(x)$ of the point $x$ in the cobweb $\Cob X$. Taking into account the arcwise connectedness of the spider $S_x$, we conclude that $S_x\subset C(x)$. To show that $C(x)\subset S_x=\pi_X^{-1}(x)$, fix any point $y\in C(x)$ and find a connected separable subspace $A\subset\Cob X$ containing the points $x$ and $y$. By Proposition~\ref{p7.1}(4), the image $\pi_X(A)$ is at most countable. Being a connected subspace of $X$, the set $\pi_X(A)$ coincides with the singleton $\{x\}$. Consequently, $y\in A\subset\pi_X^{-1}(x)=S_x$.
\end{proof}

\begin{corollary}\label{c7.5} Two premetric spaces $X,Y$ containing no countable connected subspaces are homeomorphic provided their cobwebs $\Cob X$ and $\Cob Y$ are homeomorphic.
\end{corollary}

\begin{remark} It is well-known that a connected locally connected complete metric 
space is arcwise connected, \cite[6.3.11]{En}.
The cobweb  over a connected metric space 
is not locally connected, although it is locally connected except at
a metrically discrete subset.
This illustrates how important it is
to assume that the space is locally connected at each point
if we want to conclude that it is arcwise connected.
\end{remark}

\section{The iterated cobweb construction}\label{icobweb}

In this section we shall iterate the cobweb functors and at limit obtain the functor $\Cob^\w:\PMetr\to\Metr$ assigning to each premetric space $X$ an economical complete metric space $\Cob^\w X$.

Given a premetric space $X$ put $\Cob^1 X=\Cob X$ and inductively, $\Cob^{n+1} X=\Cob(\Cob^n X)$ for $n\in\IN$. In such a way we define functors $\Cob^n:\Dist\to\Metr$ for all $n\in\IN$. For every $n\in\IN$ the spaces $\Cob^{n+1} X$ and $\Cob^n X$ are linked by the compression map $\pi_{\Cob^n X}:\Cob^{n+1} X\to\Cob^n X$. This map is non-expanding, surjective, monotone, and hereditarily quotient according to Proposition~\ref{p7.2}(4) and Theorem~\ref{t7.3}.

The iterated cobweb spaces and their compression maps form the inverse sequence 
\begin{equation}\label{invseq}\dots\to \Cob^{n+1} X\to\Cob^n X\to\dots\to\Cob^1 X.
\end{equation}
Let $$\Cob^\w X=\{(x_n)_{n\in\w}\in\prod_{n\in\IN}\Cob^n X:\forall n\in\IN\;\;\pi_{\Cob^n X}(x_{n+1})=x_n\}$$be the limit of this inverse sequence, and for every $n\in\IN$ let $$\pi^\w_n:\Cob^\w X\to \Cob^n X,\quad\pi^\w_n:(x_k)_{k\in\IN}\mapsto x_n$$ denote the limit projection. 

The space $\Cob^\w X$ is endowed with the metric
$$d_{\Cob^\w X}(a,b)=\max_{n\in\IN}\tfrac1n\cdot{d_{\Cob^n X}\big(\pi^\w_n(a),\pi^\w_n(b)\big)}$$
that generates the topology of $\Cob^\w X$ inherited from the product $\prod_{n\in\IN}\Cob^n X$. 

\begin{lemma}\label{l8.1} For every $n\in\IN$ the limit projection $\pi^\w_n:\Cob^\w X\to\Cob^n X$ is a monotone hereditarily quotient surjection.
\end{lemma}

\begin{proof} By Theorem~\ref{t7.3}(2), the bonding projections $\pi_{\Cob^n X}:\Cob^{n+1}X\to\Cob^n X$ of the inverse sequence (\ref{invseq}) are monotone hereditarily quotient surjections. Applying Theorems~9 and Corollary to Theorem 11 in \cite{Puzio}, we conclude that each limit projection $\pi^\w_n:\Cob^\w X\to\Cob^n X$ is monotone and hereditarily quotient.
\end{proof} 

Taking the composition of the projection $\pi^\w_1:\Cob^\w X\to\Cob X$ with the compression map $\pi_X:\Cob X\to X$ we get an important map $$\pi^\w_X=\pi_X\circ\pi^\w_1:\Cob^\w X\to X.$$

Some properties of this map are collected in:

\begin{theorem}\label{t8.2} For a premetric space $X$,
\begin{enumerate}
\item  $\Cob^\w X$ is an economical complete metric space of cardinality $|\Cob^\w X\,|\le|X|^\w$;
\item the map $\pi^\w_X:\Cob^\w X\to X$ is a monotone quotient surjection;
\item the map $\pi^\w_X$ is hereditarily quotient of and only if the premetric space $X$ is basic;
\item the map $\pi^\w_X:\Cob^\w X\to X$ is non-expanding provided the premetric $\bar d_X=\min\{1,d_X\}$ is a pseudometric;
\item the space $\Cob^\w X$ is connected if and only if $\Cob^\w X$ is nonseparably connected if and only if the space $X$ is connected.
\end{enumerate}
\end{theorem}

\begin{proof} 1. The completeness of the metric $d_{\Cob^\w X}$ on $\Cob^\w X$ follows from the completeness of the iterated cobwebs $\Cob^n X$ and the closedness of $\Cob^\w X$ in the Tychonov product $\prod_{n\in\IN}\Cob^n X$.

To show that $\Cob^\w X$ is economical, take any infinite subset $A\subset \Cob^\w X$ and observe that for every $n\in\IN$
\begin{equation}\label{up2}
|\pi^\w_n(A)|=|\pi_{\Cob^n X}(\pi^\w_{n+1}(A))|\le\dens(\pi^\w_{n+1}(A))\le\dens(A)
\end{equation}
according to Proposition~\ref{p7.1}(4).

Observe that for any $a,b\in A$ we have
$$d_{\Cob^\w X}(a,b)\in\big\{\tfrac1n\cdot d_{\Cob^n X}(\pi^\w_n(a),\pi^\w_n(b))\colon n\in\N\big\}$$ and thus
$$d_{\Cob^\w X}(A\times A)\subset\big\{\tfrac1n\cdot d_{\Cob^n X}(x,y):x,y\in\pi^\w_n(A),\;n\in\IN\big\}.$$
Combining this with (\ref{up2}) we get the desired inequality 
$$|d_{\Cob^\w X}(A\times A)|\le \sum_{n\in\IN}|\pi^\w_n(A)\times\pi^\w_n(A)|\le\aleph_0\cdot\dens(A)^2=\dens(A)$$
confirming the economical property of the metric $d_{\Cob^\w X}$.

Finally, we show that $|\Cob^\w X|\le|X|^\w$. This is clear if $X$ is a singleton. So, we assume that $|X|\ge 2$. It follows from the definition of cobweb space that $|\Cob X|\le  \mathfrak c\cdot |X|$. By induction, $|\Cob^n X|\le \mathfrak c\cdot|X|$ and then $|\Cob^\w X|\le(\mathfrak c\cdot|X|)^\w=|X|^\w$. 
\smallskip

2. By Lemma~\ref{l8.1} and Theorem~\ref{t7.3}(1) the maps $\pi^\w_1:\Cob^\w X\to\Cob X$ and $\pi_X:\Cob(X)\to X$ are quotient surjections. Then so is their composition  $\pi^\w_X=\pi_X\circ\pi^\w_1$.

Next, we check that the map $\pi^\w_X$ is monotone. 
Given any point $x\in X$, let $C=(\pi^\w_X)^{-1}(x)=(\pi^\w_1)^{-1}(S_x)$. Since the map $\pi^\w_1:\Cob^\w X\to\Cob X$ is hereditarily quotient, the restriction $\pi^\w_X|C:C\to S_x$ is quotient. Now the connectedness of the spider $S_x$ and Lemma~\ref{l4.4} guarantees that $C$ is connected.
\smallskip

3. If the map $\pi^\w_X=\pi_X\circ\pi^\w_1$ is hereditarily quotient, then so is the map $\pi_X:\Cob X\to X$ and then the premetric space $X$ is basic according to Theorem~\ref{t7.3}(2). On the other hand, if the premetric space $X$ is basic, then by Theorem~\ref{t7.3}(2), the compression map $\pi_X:\Cob X\to X$ is hereditarily quotient, and then the map $\pi^\w_X=\pi_X\circ\pi^\w_1$ is hereditarily quotient as the composition of two hereditarily quotient maps.
\smallskip

4. If the premetric $\bar d_X=\{1,d_X\}$ is a pseudometric, then Proposition~\ref{p7.2}(4) guarantees that  the compression map $\pi_X:\Cob X\to X$ is non-expanding. Then $\pi^\w_X=\pi_X\circ\pi^\w_1$ is non-expanding as the composition of two non-expanding maps.

5. The last item follows from the item (2), Lemma~\ref{l4.4} and Proposition~\ref{eco-p1}.
\end{proof}

\smallskip

Now we show that the construction of $\Cob^\w X$ can be completed to a functor $\Cob^\w:\Dist\to\Metr$. Given any non-expanding map $f:X\to Y$ between premetric spaces, consider the non-expanding maps $\Cob^n f:\Cob^n X\to\Cob^n Y$ for all $n\in\IN$. The naturality of the compression maps implies the commutativity of the following diagrams for all $n$:
$$\begin{CD}
\Cob^{n+1} X@>{\Cob^{n+1}f}>>\Cob^{n+1} Y\\
@V{\pi_{\Cob^n X}}VV@VV{\pi_{\Cob^n Y}}V\\
\Cob^{n}(X)@>>{\Cob^{n}f}>\Cob^n Y
\end{CD}$$
The commutativity of those diagrams ensures that the map $$\Cob^\w f:\Cob^\w X\to\Cob^\w Y,\;\;\Cob^\w f:(x_n)_{n\in\IN}\mapsto(\Cob^n f(x_n))_{n\in\IN}$$is well-defined and that the following diagram is commutative:
$$\begin{CD}
\Cob^\w X@>{\Cob^\w f}>>\Cob^\w Y\\
@V{\pi_X^\w}VV@VV{\pi_Y^\w}V\\
X@>>{f}>Y.
\end{CD}$$
The commutativity of this diagram means that the maps $\pi^\w_X:\Cob^\w X\to X$ compose components of the natural transformation  $\pi^\w:\Cob^\w\to\mathrm{Id}$ of the functor $\Cob^\w$ into the identity functor.

\section{The functor of economical resolution}\label{ecorez}

In this section we define and study the functor $\Eco:\Top\to\Metr$ of economical resolution. This functor is defined as the composition 
$$\Eco=\Cob^\w\circ \Dec$$of the functors $\Dec:\Top\to\PMetr$ and $\Cob^\w:\PMetr\to\Metr$.

Thus $\Eco X=\Cob^\w(\Dec X)$ for every topological space $X$. Next, define the resolution map $\xi_X:\Eco X\to\seq X$ as the composition $\xi_X=c_X\circ \pi^\w_{\Dec X}$ of two maps: $\pi^\w_{\Dec X}:\Cob^\w(\Dec X)\to\Dec X$ and $c_X:\Dec X\to\seq X$.
We recall that $\seq X$ stands for the sequential coreflexion of $X$ (which is $X$ endowed with the topology consisting of all sequentially open subsets).  

The maps $\xi_X$ can be seen as the components of a natural transformation  $\xi:\Eco\to \seq$ from the functor $\Eco$ to the functor $\seq$ of sequential coreflexion.

The following theorem describes some properties of the spaces $\Eco X$ and maps $\xi_X:\Eco X\to X$. Below by a {\em convergent sequence} in a topological space $X$ we understand the image $f(S_0)$ of the standard convergent sequence $S_0=\{0\}\cup\{\frac1n:n\in\IN\}\subset\IR$ under a continuous map $f:S_0\to X$.

\begin{theorem}\label{t9.1} For any topological space $X$,
\begin{enumerate}
\item $\Eco X=\Cob^\w\Dec X$ is an economical complete metric space of cardinality $|\Eco X|\le \mathfrak c\cdot|X|^\w$;
\item the map $\xi_X:\Eco X\to \seq X$ is a monotone quotient surjection;
\item the space $\Eco X$ is connected if and only if the sequential coreflexion $\seq X$ of $X$ is connected.
\item the map $\xi_X:\Eco X\to X$ is quotient if and only if the space $X$ is sequential;
\item the map $\xi_X:\Eco X\to X$ is herededitarily quotient if and only if $X$ is Fr\'echet-Urysohn;
\item each point $a\in\Eco X$ has a neighborhood $U_a\subset\Eco X$ whose image $\xi_X(U_a)$ lies in a convergent sequence in $X$.
\end{enumerate}
\end{theorem}

\begin{proof} 1. The first item follows from Theorem~\ref{t8.2}(1) and the fact that $|\Dec X|=\aleph_0\cdot|X|^\w$.
\smallskip

2. The map $\xi_X=\pi^\w_X\circ c_X:\Eco X\to\seq X$ is quotient, being the composition of two maps which are quotient by Theorem~\ref{t8.2}(2) and Corollary~\ref{c4.2}. To show that $\xi_X$ is monotone, take any point $x\in X$ and consider the preimage $c_X^{-1}(x)\subset\Dec X$. For any points $y,z\in c_X^{-1}(x)$ we get $d_{\Dec X}(y,z)=0$ by the definition of the premetric $d_{\Dec}$ and $[y,z]\subset \bar S_y$ by the definition of $\Cob(\Dec X)$. 
Now we see that the union $\bigcup_{y\in c_X^{-1}(x)}\bar S_y=(c_X\circ \pi_X)^{-1}(x)$ is a connected subspace of $\Cob(\Dec X)$. By Lemma~\ref{l8.1}, the map $\pi^\w_1:\Cob^\w X\to\Cob X$ is monotone and hereditarily quotient. Consequently, the preimage $$\xi_X^{-1}(x)=(\pi^\w_1)^{-1}(c_X\circ\pi_X)^{-1}(x))$$is connected witnessing that the map $\xi_X$ is monotone.
\smallskip

3. The third item follows from the second item and Lemma~\ref{l4.4}.
\smallskip

4. If the space $X$ is sequential, then $\seq X=X$ and hence the map $\xi_X:\Eco X\to X$ is quotient by the second statement. If the map $\xi_X:\Eco X\to X$ is quotient, then the space $X$ is sequential, being the image of a metrizable space under a quotient map, see \cite[2.4.G]{En}.
\smallskip

5. If the map $\xi_X:\Eco(X)\to X$ is hereditarily quotient, then the space $X$ is Fr\'echet-Urysohn, being the image of a metrizable space under a hereditarily quotient map, see \cite[2.4.G]{En}. 

Now assume conversely that the space $X$ is Fr\'echet-Urysohn. In this case $\seq X=X$. First we show that the composition $c_X\circ \pi_{\Dec X}:\Cob\circ \Dec X\to X$ is pseudo-open. Given any point $x\in X$ and an open set $U\subset \Cob(\Dec X)$ containing the preimage $(c_X\circ\pi_{\Dec X})^{-1}(x)$, we need to check that the image $V=c_X\circ \pi_{\Dec X}(U)$ is a neighborhood of $x$ in $X$. Assuming the converse, we could find a sequence $\{x_n\}_{n\in\IN}\subset X\setminus V$ that converges to the point $x$. Consider the continuous map $f:S_0\to X$ such that $f(0)=x_0$ and $f(\frac1n)=x_n$. Consider the point $a=(f,0)\in \Dec X$. This point also belongs to the cobweb $\Cob(\Dec X)$ and lies in the set $(c_X\circ\pi_{\Dec X})^{-1}(x)\subset U$. Since $U$ is open, there is $r\in(0,1]$ such that $B_{\Cob\Dec X}(a,r)\subset U$. Choose any integer $n>\frac1r$.  By Proposition~\ref{p7.2}(2), 
$$(f,\tfrac1n)\in B_{\Dec X}(a,r)=\pi_{\Dec X}(B_{\Cob\Dec X}(a,r))\subset \pi_{\Dec X}(U)$$ and hence  $x_n=f(\tfrac1n)=c_X(f,\frac1n)\subset c_X\circ \pi_{\Dec X}(U)=V$,
which contradicts the choice of $x_n$. This contradiction shows that the map $c_X\circ \pi_{\Dec X}$ is pseudo-open and hence hereditarily quotient.

By Lemma~\ref{l8.1}, the map $\pi^\w_1:\Cob^\w(\Dec X)\to\Cob(\Dec X)$ is hereditarily quotient. Then $\xi_X=c_X\circ \pi_{\Dec X}\circ\pi^\w_1$ is hereditarily quotient as the composition of hereditarily quotient maps.
\smallskip

6. Since $\xi_X=c_X\circ \pi_{\Dec X}\circ \pi^\w_1$ and the map $\pi^\w_1:\Cob^\w\Dec X\to\Cob\Dec X$ is continuous, it suffices to check that each point $a\in\Cob\Dec X$ has a neighborhood $U_a$ whose image $c_X\circ \pi_{\Dec X}(U_a)$ lies in a convergent sequence in $X$. If $a\in\Cob\Dec X\setminus\Dec X$, then the map $\pi_{\Dec X}$ is locally constant at $a$. Consequently, $a$ has a neighborhood $U_a$ whose image $c_X(\pi_{\Dec X}(U_a))$ is a singleton.

If $a\in\Dec X\subset \Cob\Dec X$, then $a=(f,t)$ for some continuous map $f\in C(S_0,X)$ and some $t\in S_0$. By the definition of the premetric $d_{\Dec X}$, we get $c_X(B_{\Dec X}(a,1))\subset f(S_0)$. Letting $U_a=B_{\Cob\Dec X}(a,1)$ and applying Proposition~\ref{p7.2}(2), we conclude that the image
$$c_X\circ\pi_{\Dec X}(U_a)=c_X(B_{\Dec X}(a,1))\subset f(S_0)$$lies in the convergent sequence $f(S_0)$ in $X$.
\end{proof}

\section{Locally extremal functions on connected metric spaces}

In this section we shall apply the cobweb construction to constructing a non-constant locally extremal function defined on a connected complete metric space.

A function $f:X\to\IR$ is called {\em locally extremal} if each point $x\in X$ is a point of local maximum or local minimum of $f$. In 1912 W.Sierpi\'nski \cite{Ser} proved that each continuous locally extremal function $f:\IR\to\IR$ is constant. In \cite{BGN} this result was generalized to continuous locally extremal functions $f:X\to\IR$ defined on 
a connected topological space $X$ of weight $w(X)<\mathfrak c$. Another generalization was proved by  Fedeli and Le Donne \cite{FLD}:

\begin{theorem}[Fedeli, Le Donne] A continuous locally extremal function $f:X\to \IR$ defined on a connected topological space $X$ with cellularity $c(X)<\mathfrak c$ is constant.
\end{theorem}

The cellularity requirement is essential in this theorem as shown by the projection
$\pr:[0,1]\times[0,1]\to[0,1]$ of the lexicographic square onto the interval, see \cite{BGN}. This projection is locally extremal but not constant. On the other hand, the lexicographic square $[0,1]\times[0,1]$ is a first countable linearly ordered connected compact Hausdorff space.

Having in mind this example, Morayne and W\'ojcik  asked in \cite{MWp} and \cite{MW} if there is a non-constant locally extremal function defined on a connected (complete) metric space. 

An example of such a function was first constructed by Fedeli and Le Donne \cite{FLD} and independently by the authors in \cite{BVW}. In this paper we construct such a non-constant locally extremal function with help of the cobweb construction.

Consider the set $\IA=\{-1,+1\}\times [0,1]$ endowed with the premetric
$$d((i,x),(j,y))=\begin{cases}x-y&\mbox{if $i=-1$ and $x\ge y$}\\
y-x&\mbox{if $i=+1$ and $y\ge x$}\\
1&\mbox{otherwise}.
\end{cases}
$$ 

By Theorem~\ref{t7.3}(1) the map $\pi_\IA:\Cob\,\IA\to\IA$ is a monotone quotient surjection. Consider its composition with the projection 
$$\pr:\IA\to[0,1],\;\pr:(i,x)\mapsto x.$$

\begin{theorem} The composition $f=\pr\circ\pi_\IA:\Cob\,\IA\to [0,1]$ is a surjective continuous locally extremal function defined on the  connected complete metric space $\Cob\,\IA$.
\end{theorem}

\begin{proof} It is easy to check that the premetric $d$ on $\IA$ generates the smallest topology turning the projection $\pr:\IA\to[0,1]$ into a continuous map. 
Endowed with this topology the space $\IA$ is connected. Applying Theorem~\ref{t7.3}(3) we conclude that $\Cob\,\IA$ is a connected complete metric space. 

We claim that the map $\pr\circ \pi_\IA:\Cob\,\IA\to[0,1]$ is locally extremal. Take any point $a\in \Cob\,\IA$. If $a\notin\IA$, then $f$ is locally constant at $a$ by Proposition~\ref{p7.1}(3). If $a\in\IA$, then consider the open unit ball $U_a=B_{\Cob\IA}(a,1)$ and apply Proposition~\ref{p7.2}(2) to conclude that
$$f(U_a)=\pr(\pi_\IA(B_{\Cob\IA}(a,1))=\pr(B_\IA(a,1)).$$
Let $a=(i,x)$ where $i\in\{-1,+1\}$ and $x\in[0,1]$. If $i=-1$, then $f(U_a)=\pr(B_\IA(a,r))\subset (x-r,x]$ and thus $a$ is a point of local maximum of $f$.
If $i=+1$, then $f(U_a)=\pr(B_\IA(a,r))\subset[x,x+r)$ and hence $a$ if a point of local minimum of $f$.
\end{proof}

\section{Some Open Problems}

By Corollary~\ref{c7.5}, two metric spaces $X,Y$ are homeomorphic provided their cobwebs $\Cob X$, $\Cob Y$ are homeomorphic.

\begin{problem} Are metric spaces $X,Y$ homeomorphic if so are the spaces $\Cob^\w X$ and $\Cob^\w Y$? In particular, are the spaces $\Cob^\w (\IR^n)$ and $\Cob^\w (\IR^m)$ homeomorphic for some $n\ne m$?
\end{problem}

In \cite{EPol} E.Pol constucted for every $n\in\IN$ a metric space of covering dimension $n$, having no separable subspaces of positive dimension.

\begin{problem} Let $n\in\IN$. Is there an economical complete metric space $X_n$ of covering dimension $\dim X_n=n$?
\end{problem}

A natural candidate for such a space would be $(\Cob^\w X)^n$ for a connected metric space $X$.

\begin{problem} Is $\dim (\Cob^\w X)^n=n$ for every connected metric space $X$ and every $n\in\IN$?
\end{problem}

In \cite{MW}, Morayne and W\'ojcik constructed a nonseparably connected metric group,
which is an example of a (topologically) homogeneous nonseparably connected metric space.

A metric space $X$ is called ({\em topologically}) {\em homogeneous} if for any two points $x,y\in X$ there is a bijective isometry (a homeomorphism) $f:X\to X$ with $f(x)=y$.

\begin{problem}
Can a nonseparably connected complete metric space be (topologically) homogeneous?
In particular, is there an economical connected complete metric group?
\end{problem}

\begin{problem} Is the space $\Cob^\w X$ (topologically) homogeneous for some metric space $X$ that contains more than one point?
\end{problem}

\begin{problem}
Can a nonseparably connected metric space be locally connected?
\end{problem}

\section{Acknowledgments}

We would like to thank Pawe{\l} Krupski for his topological seminar
at which we had a chance to present our work and improve it
greatly along the lines suggested by Krzysztof \  Omiljanowski. The first author expresses his thanks to Ryszard Engelking who turned his attention to Sierpi\'nski's paper \cite{Ser}.

This paper was developed over a long period of time gaining in generality and complexity of the mechanisms that we first discovered in considerably simpler settings. Therefore, an earlier draft \cite{earlier}, which is less advanced but more compact, may still be of interest to those who want to see the shortest solution of the original problem of constructing a nonseparably connected complete metric space, without studying all the mechanisms in full generality.

\end{document}